\documentclass[10pt,a4paper]{article}

\usepackage{amsmath,amsxtra,amssymb,latexsym,amscd,amsthm,array,multirow,makecell} 

\usepackage[mathcal]{euscript}

\usepackage{mathrsfs}

\usepackage{ifthen}

\usepackage[utf8]{inputenc}

\usepackage[T1]{fontenc}

\usepackage[vietnam,french,english]{babel}

\usepackage{indentfirst}

\usepackage{titlesec}

\usepackage[left=28.00mm, right=18.00mm, top=20.00mm]{geometry}

\numberwithin{equation}{section}

\usepackage{lmodern}

\allowdisplaybreaks

\usepackage{hyperref}

\def\f2{\mathbb F_2}

\def\a2{\mathcal A_2}

\def\z{\mathbb Z}

\def\u{\mathscr U}

\newcommand{\cke}[1]{\mathrm{Coker}\left(#1\right)}

\newcommand{\ex}[4]{\mathrm{Ext}_{#1}^{#2}\left(#3,#4\right)}

\newcommand{\ho}[3]{\mathrm{Hom}_{#1}\left(#2,#3\right)}

\newcommand{\ke}[1]{\mathrm{Ker}\left(#1\right)}

\newcommand{\aug}[2]{\mathrm{Aug}_{#1}\left(#2\right)}
\newcommand{\divn}[2]{\mathrm{Div}_{#1}\left(#2\right)}

\makeatletter
\newcommand{\keywords}[1]{%
  \let\@@oldtitle\@title%
  \gdef\@title{\@@oldtitle\footnotetext{\emph{Key words and phrases.} #1.}}%
}
\makeatother

\newtheorem*{theorem*}{Theorem}
\newtheorem*{algorithm*}{BG algorithm}
\newtheorem*{property*}{Property}
\newtheorem*{conj*}{Conjecture}
\newtheorem{thm}{Theorem}[section]

\newtheorem{lem}[thm]{Lemma}

\newcounter{nmdthmcnt}

    {\cleardoublepage\thispagestyle{plain}\null\begin{flushleft}%
    \Huge\bfseries Acknowledgements\end{flushleft}\vspace{1cm}\hspace{0.45cm}}%

\title{Algebraic EHP sequences}
\author{{\selectlanguage{vietnam}NGUYỄN Thế Cường}
}
\keywords{Steenrod algebra, unstable modules, EHP sequence, homotopy groups of spheres}
\date{12 December 2015}

\AtEndDocument{\bigskip{\footnotesize%
  \begin{flushleft}
  \textsc{\selectlanguage{french}Départment de Mathématiques\hfill LIAFV - CNRS\\
      LAGA - Université Paris 13\hfill Formath Vietnam\\
      99 Avenue Jean-Baptiste Clément\\
      93430 Villetaneuse} \par  
    \textit{\bfseries E-mail address}, T. C.~NGUYEN: \texttt{tdntcuong@gmail.com} or  \texttt{nguyentc@math.univ-paris13.fr}\par
    \addvspace{\medskipamount}
  \end{flushleft}}}
       
\begin{document}
\maketitle

\begin{abstract}
The James fibrations give rise to the geometric EHP sequences of homotopy groups of spheres.  Using techniques from the Lambda algebra, \cite{BCKQRS66} shows that there are similar long exact sequences of Ext groups defining the $E_{2}-$page of the Bousfield-Kan spectral sequence (also known as the unstable Adams spectral sequence) computing homotopy groups of spheres. In this paper, we give another proof for this phenomenon. 
\end{abstract}

\section{Introduction}
The James construction $J(X)$ for a connected topological space $X$ gives a model for the loop of the suspension $\Omega\Sigma X$. It allows to define the Hilton-Hopf invariants $\Omega\Sigma X\to \Omega\Sigma X^{\wedge n},$ which induce the famous theorem of Milnor and Hilton:
$$\Sigma\Omega\Sigma X\cong \Sigma\bigvee_{n\geq 1}X^{\wedge n}.$$
When $X$ is the sphere $S^{n}$, the second Hilton-Hopf invariant induces a fibration sequence after localized at prime $2$:
$$S^{n}\to \Omega\Sigma S^{n}\to \Omega\Sigma \left(S^{n}\wedge S^{n}\right).$$
The long exact sequence of homotopy groups associated to this fibration sequence is known as the EHP sequence for $S^{n}$. 

Recall that the category of unstable modules is denoted by $\u$. The most typical examples of unstable modules are modulo $2$ singular cohomology of topological spaces. Denote by $\Sigma^{n}\f2$ the reduced cohomology $\tilde{\mathrm{H}}^{*}(S^{n};\z/2)$.  The unstable Adams spectral sequence (UASS) is formulated as follows:
$$E_{2}^{s,t}(S^{n})=\ex{\u}{s}{\Sigma^{n}\f2}{\Sigma^{t}\f2}\Longrightarrow \pi_{t-s}(S^{n})^{\wedge}_{2}.$$
In \cite{BCKQRS66}, it was shown that $E_{2}^{*,*}(S^{n})$ is isomorphic to the homology of a differential module $\Lambda(n)$, obtained as a submodule of the Lambda algebra $\Lambda$. For each $n$, there is a short exact sequence:
$$0\to \Lambda(n)\to\Lambda(n+1)\to\Lambda(2n+1)\to 0.$$
Therefore there exists a long exact sequence for each $n$:
\begin{equation}\label{EHP}
\cdots\xrightarrow{H}E_{2}^{s-2,t}(S^{2n+1})\xrightarrow{P}E_{2}^{s,t}(S^{n})\xrightarrow{E}E_{2}^{s,t+1}(S^{n+1})\xrightarrow{H}E_{2}^{s-1,t}(S^{2n+1})\xrightarrow{P}\cdots\tag{EHP}
\end{equation}
and we call this the algebraic EHP sequence for $S^{n}$.

In \cite{Cuo14b}, the author gives an algorithm to compute the minimal injective resolution in the category $\u$ of $\Sigma^{t}\f2$ based on Mahowald short exact sequences. We will use this method to construct long exact sequences similar to those of \eqref{EHP}.

\section{Unstable modules and BG algorithm}
The Steenrod algebra $\mathcal{A}_{2}$ is generated by the cohomology operations $Sq^{i}$, homogeneous of degree $i\geq 0$ and the unity $Sq^{0}$, subject to the Adem relations. An unstable module $M$ is an $\f2-$graded module over the Steenrod algebra satisfying the instability condition: $\forall x\in M^{n}, Sq^{i}x=0$ as soon as $i>n$. The category $\u$ of unstable modules is an abelian category with enough injectives.  Denote by $\Sigma^{n}\f2$ the reduced singular cohomology of the sphere $S^{n}$. We write $\Sigma^{n}M$ for the tensor product $\Sigma^{n}\f2\otimes M$. The unstable injective envelope $J(n)$ of $\Sigma^{n}\f2$, called Brown-Gitler modules,  form a system of injective co-generators for $\u$. These modules are connected by Mahowald short exact sequences:
\begin{equation}\label{Mahowald}
0\to\Sigma J(n)\to J(n+1)\to J(\frac{n+1}{2})\to 0
\tag{Mahowald}
\end{equation}
where $J(n+1/2)$ is $J(k)$ if $n+1=2k$ and is trivial otherwise. Moreover, these sequences are natural in the following sense: a morphism $\Sigma J(n)\to \Sigma J(m)$ gives rise to a morphism of Mahowald exact sequences. The $J(n)$ is trivial on degrees strictly greater than $n$; $J(n)^{n}\cong \f2\left\langle x_{0}^{n}\right\rangle$ and if $x\in J(n)^{m},m\leq n$ then there exists a non-trivial Steenrod operation $\theta$ such that $\theta x=x^{n}_{0}.$ As a consequence, a morphism $J(n)\to J(m)$ is determined by a Steenrod operation of degree $n-m$. The surjection in \eqref{Mahowald} is determined 
by $Sq^{n+1/2}$ and is denoted by $\bullet Sq^{n+1/2}$.

At this point, my exposition will become more technical. I will therefore
recall some basic notations. What follows is well explained in \cite[Chapter 3]{Cuo14b}. First we will write $J(n)_{\alpha}$ with some index $\alpha$ for a direct summand $J(n)$ of an unstable module $M$ if there might be more than one copy of $J(n)$ or to distinguish this summand to a summand $J(n)$ of another unstable module. If a morphism $J(n)\to J(m)$ is determined by the Steenrod operation  $\theta$, we denote by $\bullet \theta$ this morphism. We set:
$$\aug{}{\bigoplus_{\alpha}J(n_{\alpha})}:=\bigoplus_{\alpha}J(1+n_{\alpha});\quad \divn{}{\bigoplus_{\alpha}J(n_{\alpha})}:=\bigoplus_{\alpha}J(\frac{n_{\alpha}}{2}).$$
An injective resolution $\{I^{n},\partial^{n}:I^{n}\to I^{n+1}\}_{n\geq 0}$ of $M$ is minimal if $I^{0}$ is the injective hull of $M$, $I^{1}$ is the injective hull of $I^{0}/M$ and $I^{j}$ is the injective hull of $\cke{\partial^{j-2}}$ for every $j\geq 2$.  For each non-negative integer $n$, denote by $BG(n):=\left\{BG(n)^{s},\partial_{n}^{s}:BG(n)^{s}\to BG(n)^{s+1},s\geq 0\right\}$ the minimal injective resolution of $\Sigma^{n}\f2$. Recall that each module $BG(n)^{s}$ is a finite direct sum of Brown-Gitler modules. Denote by $DA(n+1)$ the collection of injective unstable modules 
$$DA(n+1)^{s}=\aug{}{BG(n)^{s}}\bigoplus\divn{}{\aug{}{BG(n)^{s-1}}}.$$
Because of the naturality of Mahowald sequences, the morphisms $\partial_{n}^{s}$ induce:
\begin{align*}
\aug{}{BG(n)^{s}}&\xrightarrow{\delta_{1}^{s}} \aug{}{BG(n)^{s+1}},\\
\divn{}{\aug{}{BG(n)^{s}}}&\xrightarrow{\delta_{2}^{s}} \divn{}{\aug{}{BG(n)^{s+1}}},\\
\aug{}{BG(n)^{s}}&\xrightarrow{\delta_{3}^{s}} \divn{}{\aug{}{BG(n)^{s}}}.
\end{align*}
\begin{thm}[{\cite[Lemma 1.5.0.24]{Cuo14b}}]
There exist morphisms
		$$\divn{}{\aug{}{BG(n)^{s-1}}}\xrightarrow{\delta_{4}^{s-1}}\aug{}{BG(n)^{s+1}}$$
		such that
		$$\partial^{s}:=\begin{pmatrix}
		\delta_{1}^{s}&\delta_{4}^{s-1}\\
		\delta_{2}^{s-1}&\delta_{3}^{s}
		\end{pmatrix}$$ 
		make $\{DA(n+1)^{s},\partial^{s}\}$ an injective resolution of $\Sigma^{n+1}\f2$.
\end{thm}
Remark that $\delta_{3}^{s}$ is the diagonal matrix where coefficients on the diagonal are given by the surjections of Mahowald sequences. Let $J(m)$ be a direct summand of $DA(n+1)^{s}$ and let $J(k)$ be that of $DA(n+1)^{s+1}$. Then $J(m)$ is said to be  connected to $J(k)$ by $\theta$ if $\theta$ is the induced morphism 
$$J(m)\hookrightarrow DA(n+1)^{s}\xrightarrow{\partial^{s}}DA(n+1)^{s+1}\twoheadrightarrow J(k).$$
\begin{algorithm*}[{\cite[Theorem 3.1.1.1 and Lemma 3.1.1.8]{Cuo14b}}]
	\begin{enumerate}
		\item Let $J(m)_{1}$ be a direct summand of $\divn{}{\aug{}{BG(n)^{s-1}}}$ and let $J(m)_{2}$ be that of $\aug{}{BG(n)^{s+1}}$. There exist a unique direct summand $J(2m)_{3}$ of $\aug{}{BG(n)^{s-1}}$ such that $J(2m)_{3}$ is connected to $J(m)_{1}$ by $\bullet Sq^{m}$. Then $J(m)_{1}$ is connected to $J(m)_{2}$ by the identity if only if $J(2m)_{3}$ is connected to $J(m)_{2}$ by a morphism of the form $\bullet (Sq^{m}+\theta)$ for some Steenrod operation $\theta$. These are the only possible identity coefficients of $\partial^{s}$. In this case the acyclic complex 
		\begin{equation}\label{subcomp}
			J:=\{J^{r},\alpha_{r}|J^{r}=0,r\neq s,s+1;J^{s}\cong J^{s+1}\cong J(m) \}\tag{Subcomplex}
		\end{equation}
		is a sub-complex of $\{DA(n+1)^{s},\partial^{s}\}.$
		\item Set $D:=\{DA(n+1)^{s},\partial^{s}\}$. While there still exists a direct summand $J(m)_1$ of $DA(n+1)^{s}$ connecting to $J(m)_2$ of $DA(n+1)^{s+1}$ by the identity, then do $D:=D/J$ where $J$ is the acyclic complex \eqref{subcomp}.
		\item The output $D$ of the previous step is the minimal injective resolution $BG(n+1).$
	\end{enumerate}
\end{algorithm*}
\section{The construction of algebraic EHP sequences}
We index all the direct summand isomorphic to $J(n)$ of $BG(t)^{s}$ by the set $A_{(t,n,s)}$. Then:
\begin{equation}
\ex{\u}{s}{\Sigma^{n}\f2}{\Sigma^{t}\f2}\cong \ho{\u}{\Sigma^{n}\f2}{\bigoplus_{A_{(t,n,s)}}J(n)_{\alpha}}.\tag{ext}
\end{equation}
From the construction of $DA(t+1)$, the sum of all the direct summands $J(n)$ of $DA(t+1)^{s}$ is 
$$B_{(n,t+1,s)}:=C_{(n,t+1,s)}\bigoplus D_{(n,t+1,s)}$$
where 
\begin{align*}
C_{(n,t+1,s)}:=\bigoplus_{A_{(t,n-1,s)}}J(n)_{\alpha}\quad\textup{ and }\quad D_{(n,t+1,s)}:=\bigoplus_{A_{(t,2n-1,s-1)}}J(n)_{\omega}.
\end{align*}
Let $X\in \left\{C_{(n,t+1,s)},D_{(n,t+1,s)}\right\}$ and $Y\in \left\{C_{(n,t+1,s+1)},D_{(n,t+1,s+1)}\right\}$, we denote by $\partial_{X,Y}^{s}$ the composition:
$$X\hookrightarrow DA(t+1)^{s}\xrightarrow{\partial^{s}}DA(t+1)^{s+1}\twoheadrightarrow Y.$$ 
Therefore the BG algorithm shows that the maps
$$\partial_{DA}^{s}:=\begin{pmatrix}
\partial_{C_{(n,t+1,s)},C_{(n,t+1,s+1)}}^{s} & \partial_{D_{(n,t+1,s)},C_{(n,t+1,s+1)}}^{s}\\
\partial_{C_{(n,t+1,s)},D_{(n,t+1,s+1)}}^{s} & \partial_{D_{(n,t+1,s)},D_{(n,t+1,s+1)}}^{s}
\end{pmatrix}=\begin{pmatrix}
0 & \partial_{D_{(n,t+1,s)},C_{(n,t+1,s+1)}}^{s}\\
0 & 0
\end{pmatrix}$$
make $\left\{B_{(n,t+1,s)},\partial_{DA}^{s}\right\}$ a complex. 
Hence
\begin{align*}
\ex{\u}{s}{\Sigma^{n}\f2}{\Sigma^{t+1}\f2}&\cong \mathrm{H}_{s}\ho{\u}{\Sigma^{n}\f2}{DA(t+1)},\\
&\cong \mathrm{H}_{s}\ho{\u}{\Sigma^{n}\f2}{B_{(n,t+1,s)}},\\
&\cong \ke{\left(\partial_{D_{(n,t+1,s)},C_{(n,t+1,s+1)}}^{s}\right)_{*}}\bigoplus \cke{\left(\partial_{D_{(n,t+1,s-1)},C_{(n,t+1,s)}}^{s-1}\right)_{*}}.
\end{align*}
We will now explicit the morphisms $\left(\partial_{D_{(n,t+1,s)},C_{(n,t+1,s+1)}}^{s}\right)_{*}$. The source of this morphism is described as follows:
\begin{align*}
\ho{\u}{\Sigma^{n}\f2}{D_{(n,t+1,s)}}&\cong \ho{\u}{\Sigma^{n}\f2}{\bigoplus_{A_{(t,2n-1,s-1)}}J(n)_{\omega}},\\
&\cong \ho{\u}{\Sigma^{2n-1}\f2}{\bigoplus_{A_{(t,2n-1,s-1)}}J(2n-1)_{\omega}},\\
&\cong \ex{\u}{s-1}{\Sigma^{2n-1}\f2}{\Sigma^{t}\f2}.
\end{align*}
And the target is:
\begin{align*}
\ho{\u}{\Sigma^{n}\f2}{C_{(n,t+1,s+1)}}&\cong \ho{\u}{\Sigma^{n}\f2}{\bigoplus_{A_{(t,n-1,s+1)}}J(n)_{\alpha}},\\
&\cong \ho{\u}{\Sigma^{n-1}\f2}{\bigoplus_{A_{(t,n-1,s+1)}}J(n-1)_{\alpha}},\\
&\cong \ex{\u}{s+1}{\Sigma^{n-1}\f2}{\Sigma^{t}\f2}.
\end{align*}
Therefore 
$$\left(\partial_{D_{(n,t+1,s)},C_{(n,t+1,s+1)}}^{s}\right)_{*}: \ex{\u}{s-1}{\Sigma^{2n-1}\f2}{\Sigma^{t}\f2}\to \ex{\u}{s+1}{\Sigma^{n-1}\f2}{\Sigma^{t}\f2}.$$
Denote by $P^{s}$ the morphisms $\left(\partial_{D_{(n,t+1,s)},C_{(n,t+1,s+1)}}^{s}\right)_{*}$ then we obtain the following exact sequences:
$$0\to\ke{P^{s-1}}\to \ex{\u}{s-2}{\Sigma^{2n+1}\f2}{\Sigma^{t}\f2}\to \ex{\u}{s}{\Sigma^{n}\f2}{\Sigma^{t}\f2}\to \cke{P^{s-1}}\to 0,$$
$$0\to \cke{P^{s-1}}\to \ex{\u}{s}{\Sigma^{n+1}\f2}{\Sigma^{t+1}\f2}\to \ke{P^{s}}\to 0.$$
We can now conclude:
\begin{thm}[Algebraic EHP sequences]
	There exist a long exact sequence for each $n$:
	\begin{equation}\label{EHP}
	\cdots\xrightarrow{H}E_{2}^{s-2,t}(S^{2n+1})\xrightarrow{P}E_{2}^{s,t}(S^{n})\xrightarrow{E}E_{2}^{s,t+1}(S^{n+1})\xrightarrow{H}E_{2}^{s-1,t}(S^{2n+1})\xrightarrow{P}\cdots\tag{EHP}
	\end{equation}
	and we call this the algebraic EHP sequence for $S^{n}$.
\end{thm}
\section{A special case}
This section studies the special case of EHP sequence for $S^{2^{k}-1}$. In particular we show that in this case the morphisms $P$ are trivial. 
\begin{lem}
    If $n=2^{k}$ then the morphism $\left(\partial_{D_{(n,t+1,s)},C_{(n,t+1,s+1)}}^{s}\right)_{*}$ is trivial. 	
\end{lem}
\begin{proof}
	Recall that 
	$$\partial_{D_{(2^{k},t+1,s)},C_{(2^{k},t+1,s+1)}}^{s}:\bigoplus_{A_{(t,2^{k+1}-1,s-1)}}J(2^{k})_{\omega}\to \bigoplus_{A_{(t,2^{k}-1,s+1)}}J(2^{k})_{\alpha}.$$
	Suppose that this morphism is not trivial then there must exist an identity coefficient in the matrix form. The first item in the definition of the BG algorithm assures that in this case, the Steenrod square $Sq^{2^{k}}$ can be factorized as a product of other Steenrod operations. This is obviously false since $Sq^{2^{k}}$ is indecomposable. 	
	Therefore $\partial_{D_{(2^{k},t+1,s)},C_{(2^{k},t+1,s+1)}}^{s}$ is trivial.
\end{proof}
We get the following consequence:
\begin{thm}[James splitting]
	There are short exact sequences
	$$0\to E^{s,t}_{2}(S^{2^{k}-1})\to E^{s,t+1}_{2}(S^{2^{k}})\to E^{s-1,t}_{2}(S^{2^{k+1}-1})\to 0$$
	for every $k$.
\end{thm}
\bibliographystyle{alpha}
\bibliography{Thuvien}
\end{document}